\documentclass[12pt,verbatim]{amsart}
\usepackage{amsfonts}
\usepackage{ifthen}
\usepackage{amsthm}
\usepackage{amsmath}
\usepackage{amssymb}
\usepackage{graphicx}
\usepackage{amscd,amssymb,amsthm}

\usepackage{caption}

\newcounter{minutes}
\divide\time by 60
\newcounter{hours}
\multiply\time by 60 \addtocounter{minutes}{-\time}

\title{The adjacent sides of hyperbolic Lambert quadrilaterals}

\author{Gendi Wang}

\address{School of Science, Zhejiang Sci-Tech University, Hangzhou 310018, China}
\email{gendi.wang@zstu.edu.cn}

\newcommand{\comment}[1]{}

\newtheorem{theorem}[equation]{Theorem}

\newtheorem{lemma}[equation]{Lemma}
\newtheorem{proposition}[equation]{Proposition}

\newtheorem{corollary}[equation]{Corollary}
\newtheorem{remark}[equation]{Remark}

\newcommand{\beq}{\begin{equation}}
\newcommand{\eeq}{\end{equation}}
\newcommand{\B}{\mathbb{B}^2}

\newcommand{\C}{\mathbb{C}}
\newcommand{\UH}{\mathbb{H}^2}

\newcommand{\arth}{{\rm arth}}
\newcommand{\arsh}{{\rm arsh}}

\numberwithin{equation}{section}

\begin{document}

\def\thefootnote{}
\footnotetext{ \texttt{\tiny File:~\jobname .tex,
          printed: \number\year-\number\month-\number\day,
          \thehours.\ifnum\theminutes<10{0}\fi\theminutes}
} \makeatletter\def\thefootnote{\@arabic\c@footnote}\makeatother

\maketitle

\begin{abstract}
We prove sharp bounds for the product and the sum of the hyperbolic lengths of a pair of hyperbolic adjacent sides of hyperbolic Lambert quadrilaterals in the unit disk. We also show the H\"older convexity of the inverse hyperbolic sine function involved in the hyperbolic geometry.
\end{abstract}


{\small \sc Mathematics Subject Classification~(2010).} {51M09~(26D07)}

\section{Introduction}

Given a pair of points in the closure of the unit disk ${\mathbb B}^2$ in the complex plane $\C$, there exists a unique hyperbolic geodesic line joining these two points.
Hyperbolic lines are simply sets of the form $C \cap {\mathbb B^2}$ where $C$ is a circle perpendicular to the unit circle, or a  Euclidean diameter of
$\B\,.$ For a quadruple of four points
$\{a,b,c,d\}$ in the closure of the unit disk we can draw these hyperbolic lines
through each of the four pairs of points $\{a,b\},$ $\{b,c\},$ $\{c,d\},$ and $\{d,a\}\,.$ If these hyperbolic lines bound a domain $D \subset {\mathbb B}^2$  such that the points $\{a,b,c,d\}$ are in the positive order on the boundary of the domain $D\,,$ then we say
that the quadruple of points $\{a,b,c,d\}$ determines a hyperbolic quadrilateral $Q(a,b,c,d)\,$ and that the points $a,b,c,d$ are its vertices.
A hyperbolic quadrilateral with angles equal to $\pi/2, \pi/2, \pi/2, \phi\,(0\leq \phi < \pi/2)\,,$ is
called a hyperbolic {\em Lambert} quadrilateral \cite[p.~156]{be}, see Figure 1. Observe that one of the vertices of a Lambert quadrilateral may be on the unit circle, in which case the angle at that vertex is $\phi=0\,.$

In \cite[Theorems 1.1 and 1.2]{vw}, the authors gave the sharp bounds of the product and the sum of two hyperbolic distances between the opposite sides of hyperbolic Lambert quadrilaterals in the unit disk. By \cite[Proposition 3.3]{vw}, we know that the above two hyperbolic distances are also a pair of the lengths of the adjacent sides of hyperbolic Lambert quadrilaterals with respect to the vertex $v_a$ (see $d_1$, $d_2$ in Figure 1). Therefore, it is natural to raise the problem: how does the pair of the lengths of the adjacent sides with respect to the vertex $v_c$ behave (see $d_3$, $d_4$ in Figure 1)? This is the motivation of this paper, and we will find the sharp bounds of the product and the sum of this pair of the lengths of the adjacent sides of hyperbolic Lambert quadrilaterals in the unit disk. The main results of this paper are formulated as follows.

\begin{theorem}\label{Lamdd}
Let $Q (v_a\,,v_b\,,v_c\,,v_d)$ be a hyperbolic Lambert quadrilateral in $\B$ and let the quadruple of interior angles
$(\frac{\pi}{2}\,,\frac{\pi}{2}\,,\phi\,,\frac{\pi}{2})$, $\phi\in[0, \pi/2)\,,$ corresponds to the quadruple $(v_a\,,v_b\,,v_c\,,v_d)$ of vertices. Let
$d_3=\rho(v_c\,,v_b)\,,$  $d_4=\rho(v_c\,,v_d)$ (see Figure \ref{Lamb}), and
let $s={\rm th} \rho(v_a,v_c)\in(0,1)$, where $\rho$ is the hyperbolic metric in the unit disk.
Then
\begin{eqnarray*}
d_3d_4\leq \left(\log \sqrt{\frac{1+s\sqrt{2-s^2}}{1-s^2}}\right)^2
\end{eqnarray*}
Here equality holds if and only if $v_c$ is on the bisector of the interior angle at $v_a$.
\end{theorem}

\begin{theorem}\label{Lamdad}
Let $Q(v_a\,,v_b\,,v_c\,,v_d)$\,, $d_3$, $d_4$ and $s$ be as in Theorem \ref{Lamdd}. Then
\begin{eqnarray*}
\log \sqrt{\frac{1+s}{1-s}}<d_3+d_4\le 2\log \frac{1+s\sqrt{2-s^2}}{1-s^2}.
\end{eqnarray*}
Equality holds in the right-hand side if and only if $v_c$ is on the bisector of the interior angle at $v_a$.
\end{theorem}

\begin{figure}[h]
\centering
\includegraphics[width=7.5cm]{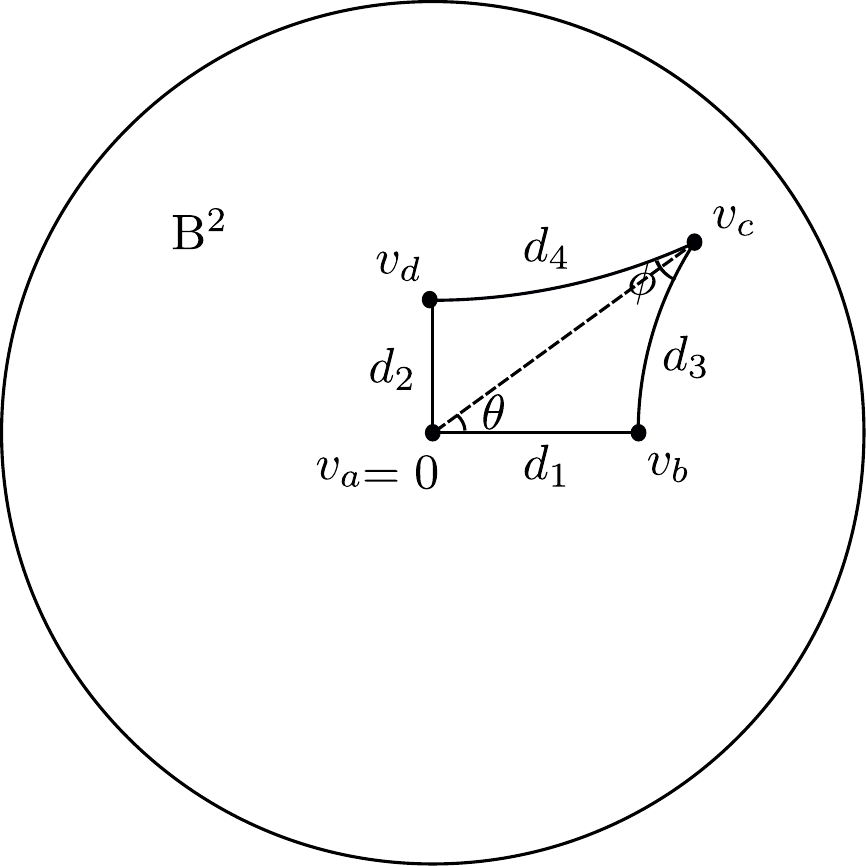}
\caption{\label{Lamb}A hyperbolic Lambert quadrilateral in $\B$}
\end{figure}

We denote the other two sides of $Q(v_a\,,v_b\,,v_c\,,v_d)$ by
$$d_1=\rho(v_a\,,v_b)\,\,\,\,{\rm and}\,\,\,\,d_2=\rho(v_a\,,v_d).$$

In a Lambert quadrilateral, the angle $\phi$ is related to the lengths
$d_1$, $d_2$ of the sides "opposite" to it as follows
\cite[Theorem 7.17.1]{be}:
$${\rm sh}\, d_1{\rm sh}\, d_2=\cos\phi.$$
See also the recent paper of A.~F.~Beardon and D.~Minda \cite[Lemma 5]{bm}.

In \cite[Corollary 1.3]{vw}, M. Vuorinen and G.-D. Wang provided a connection between $d_1$, $d_2$ and $s={\rm th} \rho(v_a,v_c)$ as follows
\begin{eqnarray}\label{d1d2th}
{\rm th}^2\,d_1+{\rm th}^2\,d_2=s^2.
\end{eqnarray}

Proposition \ref{d3d4for} (in section 2) yields the following corollary, which provides a connection between $d_3$, $d_4$ and $s={\rm th} \rho(v_a,v_c)$.

\begin{corollary} Let $s$, $d_3$ and $d_4$ be as in Theorem \ref{Lamdd}. Then
\begin{eqnarray}\label{d3d4sh}
{\rm sh}^2\,d_3+{\rm sh}^2\,d_4=\frac{s^2}{1-s^2}.
\end{eqnarray}
\end{corollary}

By \eqref{d1d2th} and \eqref{d3d4sh}, we get the following equality
$$\frac{1}{{\rm th}^2\,d_1+{\rm th}^2\,d_2}-\frac{1}{{\rm sh}^2\,d_3+{\rm sh}^2\,d_4}=1,$$
which shows the relation between the four sides of the hyperbolic Lambert quadrilateral $Q(v_a\,,v_b\,,v_c\,,v_d)$.

This paper is organized as follows. In Section 2, the notation and facts on the hyperbolic metric are stated, and some lemmas on the inverse hyperbolic trigonometric functions are proved. Section 3 is devoted to the H\"older convexity of the inverse hyperbolic sine function. The main results are proved in Section 4.

\section{Preliminaries}

It is assumed that the reader is familiar with basic definitions of geometric
function theory, see e.g. \cite{be,v}.
We recall here some basic information on hyperbolic geometry \cite{be}.

The chordal distance is defined by
\beq\label{q}
\left\{\begin{array}{ll}
q(x,y)=\frac{|x-y|}{\sqrt{1+|x|^2}\sqrt{1+|y|^2}},&\,\,\, x\,,y\neq\infty\\
q(x,\infty)=\frac{1}{\sqrt{1+|x|^2}},&\,\,\, x\neq\infty,
\end{array}\right.
\eeq
for $x,y\in\overline\C$.

For an ordered quadruple $a,b,c,d$ of distinct points in $\overline{\C}$ we define the absolute ratio by
$$|a,b,c,d|=\frac{q(a,c)q(b,d)}{q(a,b)q(c,d)}.$$
It follows from (\ref{q}) that for distinct points $a,b,c,d\in \C$
\beq\label{crossratio}
|a,b,c,d|=\frac{|a-c||b-d|}{|a-b||c-d|}.
\eeq
The most important property of the absolute ratio is M\"obius invariance, see \cite[Theorem 3.2.7]{be}, i.e., if $f$ is a M\"obius transformation, then
$$|f(a),f(b),f(c),f(d)|=|a,b,c,d|,$$
for all distinct $a,b,c,d$ in $\overline{\C}$.

For a domain $G\varsubsetneq \C$ and a continuous weight function $w: G\rightarrow(0,\infty)\,,$  we define the weighted length of a rectifiable curve $\gamma\subset G$ to be
$$l_w(\gamma)=\int_{\gamma}w(z)|dz|$$
and the weighted distance between two points $x,y \in G $ by
$$d_w(x,y)=\inf_{\gamma}l_w(\gamma),$$
where the infimum is taken over all rectifiable curves in $G$ joining $x=x_1+i x_2$ and $y=y_1+i y_2$. It is easy to see that $d_w$ defines a metric on $G$ and $(G,d_w)$ is a metric space. We say that a curve $\gamma: [0,1]\rightarrow G$ is a geodesic joining $\gamma(0)$ and $\gamma(1)$ if for all $t\in (0,1)$, we have
$$d_w(\gamma(0),\gamma(1))=d_w(\gamma(0),\gamma(t))+d_w(\gamma(t),\gamma(1)).$$
The hyperbolic distance in the upper half plane $\UH=\{x=x_1+i x_2\in\C\,|\,x_2>0\}$ and the unit disk $\B$ is defined in terms of the weight functions
$w_{\UH}(x)=1/{x_2}$ and  $w_{\B}(x)=2/{(1-|x|^2)}\,,$ respectively.  We also have the corresponding explicit formulas
\beq\label{cosh}
\cosh\rho_{\UH}(x,y)=1+\frac{|x-y|^2}{2x_2y_2}
\eeq
for all $x,y\in \UH$  \cite[p.35]{be}, and
\beq\label{sh}
\rho_{\B}(x,y)=2\arsh\,\frac{|x-y|}{\sqrt{(1-|x|^2)(1-|y|^2)}}
\eeq
for all $x,y\in \B$  \cite[p.40]{be}. In particular, for $t\in(0,1)$,
\beq\label{arth}
\rho_{\B}(0,t)=\log\frac{1+t}{1-t}=2\arth\,t.
\eeq

There is a third equivalent way to express the hyperbolic distances.
Let $G\in\{\UH,\B\}$, $x,y\in{G}$ and let $L$ be an arc of a circle
perpendicular to $\partial G$ with $x,y\in L$ and let
$\{x_*,y_*\}=L\cap\partial G$, the points being labeled so that $x_*, x, y, y_*$ occur in this order on $L$. Then by \cite[(7.26)]{be}
\beq\label{rho}
\rho_G(x,y)=\sup\{\log|a,x,y,b|:a,b\in\partial G\}=\log|x_*,x,y,y_*|.
\eeq

The hyperbolic distance remains invariant under M\"obius transformations of $G$ onto $G'$ for $G,\,G'\in\{\UH,\B\}$.

Hyperbolic geodesics are  arcs of circles which are orthogonal to the boundary of the domain. More precisely, for $a,b\in \B$ (or $\UH)$, the hyperbolic geodesic segment joining $a$ to $b$ is an arc of a circle orthogonal to $S^1$ (or $\partial \UH)$. In a limiting case the points $a$ and $b$ are located on a Euclidean line through $0$ (or located on a normal of $\partial \UH$), see \cite{be}.  Therefore, the points $x_*$ and $y_*$ are the end points of the hyperbolic geodesic. For any two distinct points the hyperbolic geodesic segment is unique (see Figure \ref{h2} and \ref{b2}). For basic facts about the hyperbolic geometry we suggest the interested readers to refer \cite{a}, \cite{be} and \cite{kl}.

\medskip
\begin{figure}[h]
\begin{minipage}[t]{0.45\linewidth}
\centering
\includegraphics[width=8cm]{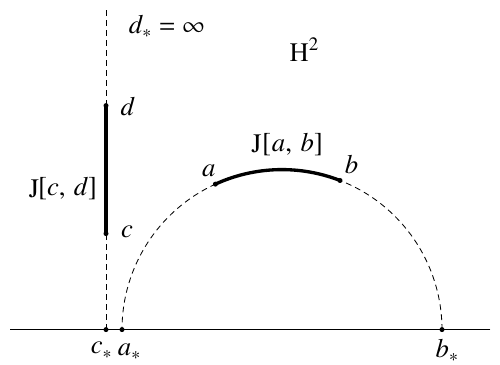}
\caption{\label{h2} Hyperbolic geodesic segments in $\UH$}
\end{minipage}
\hfill
\hspace{1cm}
\begin{minipage}[t]{0.45\linewidth}
\centering
\includegraphics[width=6cm]{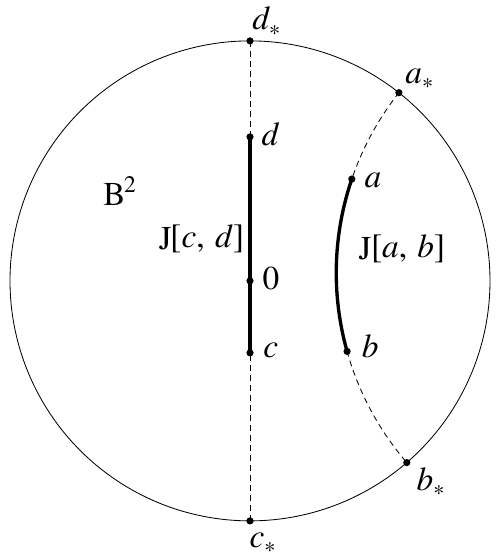}
\caption{\label{b2} Hyperbolic geodesic segments in $\B$}
\end{minipage}
\end{figure}
\medskip

By \cite[Exercise 1.1.27]{k} and \cite[Lemma 2.2]{kv}, for  $x\,,y \,\in \mathbb{B}^2\setminus\{0\}$ such that $0,x,y$ are noncollinear,
the circle $S^1(a,r_a)$ containing $x, y$ is orthogonal to the unit circle, where
\beq\label{orar}
a =i\frac{y(1+|x|^2)-x(1+|y|^2)}{2(x_2y_1-x_1y_2)}\,\,\,\, and\, \,\,\,r_a=\frac {|x-y|\big|x|y|^2-y\big|}{2|y||x_1y_2-x_2y_1|}\,.
\eeq

\medskip

The following {\em
monotone form of l'H${\rm \hat{o}}$pital's rule} is useful in deriving monotonicity properties and obtaining inequalities. See the extensive bibliography of \cite{avz}.

\begin{lemma} \label{lhr}{\rm \cite[Theorem 1.25]{avv1}}.
For $-\infty<a<b<\infty$, let $f,\,g: [a,b]\rightarrow \mathbb{R}$ be continuous on $[a,b]$, and be differentiable on $(a,b)$, and let $g'(x)\neq 0$ on $(a,b)$. If $f'(x)/g'(x)$ is increasing(deceasing) on $(a,b)$, then so are
\begin{eqnarray*}
\frac{f(x)-f(a)}{g(x)-g(a)}\,\,\,\,\,\,\,and\,\,\,\,\,\,\,\,\frac{f(x)-f(b)}{g(x)-g(b)}.
\end{eqnarray*}
If $f'(x)/g'(x)$ is strictly monotone, then the monotonicity in the conclusion is also strict.
\end{lemma}

From now on we let $r'=\sqrt{1-r^2}$ for $0<r<1$.

\begin{lemma}\label{lecr}
Let $s\in(0,1)$, $m={s}/\sqrt{1-s^2}$ and $r\in(0,1)$.

(1) The function $f_1(r)\equiv\arsh(m\,r)/ \arth(s\,r)$ is strictly decreasing with range $(1,1/\sqrt{1-s^2})$.

(2) The function $f_2(r)\equiv\arsh(m\,r) \arsh(m\,r')$ is strictly increasing on $(0,\frac{\sqrt{2}}{2}]$ and strictly decreasing on $[\frac{\sqrt{2}}{2}, 1)$ with maximum value $(\arsh(\frac{\sqrt{2}}{2} m))^2$.

(3) The function $f_3(r)\equiv\arsh(m\,r)+\arsh(m\,r')$ is strictly increasing on $(0,\frac{\sqrt{2}}{2}]$ and strictly decreasing on $[\frac{\sqrt{2}}{2}, 1)$ with range $(\arsh\,m,\,2\arsh\,(\frac{\sqrt{2}}{2} m)]$.
\end{lemma}

\begin{proof}
(1) Let $f_{11}(r)=\arsh(m\,r)$ and $f_{12}(r)=\arth(s\,r)$. Then $f_{11}(0^+)=f_{12}(0^+)=0$. By differentiation,
$$
\frac{f'_{11}(r)}{f'_{12}(r)}=\frac{1-s^2r^2}{\sqrt{1-s^2r'^2}}
$$
which is strictly decreasing. Hence by Lemma \ref{lhr}, $f_1$ is strictly decreasing with $f_1(0^+)=1/\sqrt{1-s^2}$ and $f_1(1^-)=1$ because
$$\arsh\,m=\log\sqrt{\frac{1+s}{1-s}}=\arth\,s.$$

(2) By differentiation,
$$f'_2(r)=\frac{m}{r'\sqrt{1+m^2r'^2}\sqrt{1+m^2r^2}}(\phi_2(r')-\phi_2(r)),$$
where $\phi_2(r)=r\sqrt{1+m^2r^2}\arsh(m\,r)$. It is clear that $\phi_2$ is strictly increasing. Therefore, we get the result.

(3) By differentiation,
$$f'_3(r)=\frac{m}{r'\sqrt{1+m^2r'^2}\sqrt{1+m^2r^2}}(\phi_3(r')-\phi_3(r)),$$
where $\phi_3(r)=r\sqrt{1+m^2r^2}$. It is clear that $\phi_3$ is strictly increasing and hence the result follows immediately.
\end{proof}

\begin{proposition}\label{d3d4for}
Let $Q (v_a\,,v_b\,,v_c\,,v_d)$ be a hyperbolic Lambert quadrilateral in $\B$ and let the quadruple of interior angles
$(\frac{\pi}{2}\,,\frac{\pi}{2}\,,\phi\,,\frac{\pi}{2})$, $\phi\in[0, \pi/2)\,,$ corresponds to the quadruple $(v_a\,,v_b\,,v_c\,,v_d)$ of vertices. Let $d_1=\rho(v_a\,,v_b)\,,$  $d_2=\rho(v_a\,,v_d)$
$d_3=\rho(v_c\,,v_b)\,,$  $d_4=\rho(v_c\,,v_d)$, and
let $s={\rm th} \rho(v_a,v_c)\in(0,1)$ and $m=\frac{s}{\sqrt{1-s^2}}$.
Then
$$d_1=\arth(s\,r),\,\,\,\,\,\,\,\,d_2=\arth(s\,r'),\,\,\,\,\,\,\,\,d_3=\arsh(m\,r'),\,\,\,\,\,\,\,\,d_4=\arsh(m\,r),$$
where $r\in(0,1)$ is a constant.
\end{proposition}
\begin{proof}
Since the hyperbolic distance is M\"obius invariant, we may assume that $v_a=0$, $v_b$ is on the real axis $X$, $v_d$ is on the imaginary axis $Y$ and $v_c=t e^{i\theta}$, $0<t<1$ and $0<\theta<\frac{\pi}{2}$ (see Figure \ref{Lamb}). Let $r=\cos\theta$.
Then by (\ref{orar}) the circle $S^1(b,r_b)$ through $v_c$ and $\overline{v_c}$ is orthogonal to $\partial\B$, where
$$
b=\frac{1+t^2}{2t\cos\theta}\,\,\,\,\,{\rm and}\,\,\,\,r_b=\frac{\sqrt{(1+t^2)^2-4t^2\cos^2\theta}}{2t\cos\theta}.
$$
The distances $d_1$ and $d_2$ can either be obtained from the proof of Theorem 1.1 in \cite{vw} or be directly derived by \eqref{arth}.

By \eqref{sh}, we get
$$
d_3=\rho(v_c,v_b)=2\arsh\,f_s(r),
$$
where
$$f_s(r)=\sqrt{\frac{g^2_s(r)+g^2_s(1)-2 r g_s(r)g_s(1)}{(1-g^2_s(r))(1-g^2_s(1))}}$$
and
$$g_s(r)=\frac{1-\sqrt{1-s^2r^2}}{s\,r}.$$
Similarly, we get
$$
d_4=\rho(v_c,v_d)=2\arsh\,f_s(r').
$$
By calculation, we get the equality
$$f_s(r)\sqrt{1+f^2_s(r)}=\frac{m r'}{2},$$
which implies ${\rm sh}d_3=m r'$ and ${\rm sh}d_4=m r$. Therefore,
$$d_3=\arsh(m\,r')\,\,\,\,\,\,{\rm and}\,\,\,\,\,\,d_4=\arsh(m\,r).$$
\end{proof}

By Proposition \ref{d3d4for} and Lemma \ref{lecr}(1), we have the following theorem.
\begin{theorem}\label{Lamd2d3}
Let $Q(v_a\,,v_b\,,v_c\,,v_d)$\,,  $d_1$, $d_2$, $d_3$, $d_4$ and $s$ be as in Proposition \ref{d3d4for}.
Then
$$d_2<d_3<\frac{1}{\sqrt{1-s^2}}d_2\,\,\,\,{\rm and}\,\,\,\,d_1<d_4<\frac{1}{\sqrt{1-s^2}}d_1.$$
\end{theorem}

\begin{remark}
M. Vuorinen and G.-D. Wang gave the bounds for the product and sum of $d_1$ and $d_2$ \cite[Theorem 1.1 and Theorem 1.2]{vw}, by which and Theorem \ref{Lamd2d3}, we can get the bounds for the product and sum of $d_3$ and $d_4$. But the results are weaker than that of Theorem \ref{Lamdd} and Theorem \ref{Lamdad} in this paper.
\end{remark}

\section{The H\"older convexity for the inverse hyperbolic sine function}

The inverse hyperbolic sine and tangent functions play important roles in the study of the hyperbolic metric. In \cite[Theorem 2.21]{vw}, the authors showed the H\"older convexity of the inverse hyperbolic tangent function.  We will study the similar property of the inverse hyperbolic sine function in this section.

For $r,s\in(0,+\infty)$, the \emph{H\"older mean of order $p$} is defined by
$$H_p(r,s)=\left(\frac{r^p+s^p}{2}\right)^{1/p} \quad\mbox{for}\quad{p\neq 0},
\quad H_0(r,s)=\sqrt{r\,s}.$$
For $p=1$, we get the arithmetic mean $A=H_1$; for $p=0$, the
geometric mean $G=H_0$; and for $p=-1$, the harmonic mean $H=H_{-1}$.
It is well-known that $H_p(r,s)$ is continuous and increasing with respect to $p$ \cite[P.203, Theorem 1]{bu}.
Many other interesting properties of H\"older means are given in \cite{bu} and \cite{hlp}.

A function $f:I\to J$ is called {\it $H_{p,q}$-convex(concave)} if it satisfies
   $$f(H_p(r,s))\leq(\geq)H_q(f(r),f(s))$$
for all $r,s\in I$, and  {\it strictly $H_{p,q}$-convex(concave)} if the
inequality is strict except for $r=s$. For $H_{p,q}$-convexity
of some special functions the reader is referred to \cite{avv2, avz, ba1, ba, cwzq, zwc}.

\begin{lemma}\label{t1l1}
Let $r\in(0,+\infty)$.

(1) The function $f_1(r)\equiv\frac{\arsh\,r}{r}$ is strictly decreasing with range $(0,1)$.

(2) The function $f_2(r)\equiv\frac{r(1+r^2)-\sqrt{1+r^2}\,\arsh\,r}{r^3}$ is strictly increasing with range $(2/3,1)$.
\end{lemma}
\begin{proof}
(1) Let $f_{11}(r)=\arsh\,r$ and $f_{12}(r)=r$. It is easy to see that $f_{11}(0^+)=f_{12}(0^+)=0$, then
$$
\frac{f'_{11}(r)}{f'_{12}(r)}=\frac{1}{\sqrt{1+r^2}}
$$
which is strictly decreasing. Hence by Lemma \ref{lhr}, $f_1$ is strictly decreasing with $f_1(0^+)=1$ and $f_1(+\infty)=\lim\limits_{r\rightarrow +\infty}f_1(r)=\lim\limits_{r\rightarrow +\infty}\frac{f'_{11}(r)}{f'_{12}(r)}=0$.

\medskip
(2) Let $f_{21}(r)=r(1+r^2)-\sqrt{1+r^2}\,\arsh\,r$ and $f_{22}(r)=r^3$. Then $f_{21}(0^+)=f_{22}(0^+)=0$.
By differentiation, we have
$$
\frac{f'_{21}(r)}{f'_{22}(r)}=1-\frac 13 \frac{1}{\sqrt{1+r^2}}\frac{\arsh\,r}{r},
$$
which is strictly increasing by (1). Hence by Lemma \ref{lhr}, $f_2$ is strictly increasing with $f_2(0^+)=2/3$ and $f_2(+\infty)=\lim\limits_{r\rightarrow +\infty}f(r)=1$.
\end{proof}

\begin{lemma}\label{t1l2}
For $p \in \mathbb{R}$ and $r\in(0,+\infty)$ define
$$h_p(r)\equiv 1+p \sqrt{1+r^2}\frac{\arsh\,r}{r}-\frac{1}{\sqrt{1+r^2}}\frac{\arsh\,r}{r}.$$

(1) If $p\le-2$, then $h_p$ is strictly decreasing with range $(-\infty,p)$.

(2) If $p>0$, then $h_p$ is strictly increasing with range $(p,+\infty)$.

(3) If $p=0$, then $h_p$ is strictly increasing with range $(0,1)$.

(4) If $-2<p<0$, then the range of $h_p$ is $(-\infty, C(p)]$,
where $C(p)\equiv\sup\limits_{0<r<+\infty}{h_p(r)}\in (p\,,1)$. Moreover, $\lim\limits_{p\to-2}C(p)=-2$ and $\lim\limits_{p\to 0}C(p)=1$.
\end{lemma}

\begin{proof}
By l'H\^opital's Rule, we get
$$
\lim\limits_{r\to+\infty}\frac{\sqrt{1+r^2}\arsh\,r}{r}=\lim\limits_{r\to+\infty}\left(1+\frac{r\,\arsh\,r}{\sqrt{1+r^2}}\right)=+\infty.
$$
Together with Lemma \ref{t1l1} (1), we have $h_p(0^+)=p$ and
\beq
h_p(+\infty)=\lim\limits_{r\rightarrow +\infty}h_p(r)=
\left\{\begin{array}{ll}
-\infty &\,\,\,p<0,\\
1 &\,\,\,p=0,\\
+\infty &\,\,\,p>0.
\end{array}\right.
\eeq

Next by differentiation, we have
$$
h'_p(r)=\frac 1r\left(1-\frac{1}{\sqrt{1+r^2}}\frac{\arsh\, r}{r}\right)[p-f(r)],
$$
where $f(r)=2-\frac{1}{1+r^2}-\frac{2}{f_2(r)}$ and $f_2(r)$ is as in Lemma \ref{t1l1} (2). Therefore, $f$ is strictly increasing from $(0,+\infty)$ onto $(-2,0)$. Hence we get (1)-(3).

(4) If $-2<p<0$, since the range of $f$ is $(-2,0)$, we see that
there exists exactly one point $r_0\in(0,+\infty)$ such that $p=f(r_0)$. Then $h_p$ is increasing on $(0, r_0)$ and decreasing on $(r_0,+\infty)$. Since
$$
h_p(r)=1-\left(-p \sqrt{1+r^2}\frac{\arsh\,r}{r}+\frac{1}{\sqrt{1+r^2}}\frac{\arsh\,r}{r}\right)<1,
$$
by the continuity of $h_p$, there is a continuous function
$$C(p)\equiv\sup\limits_{0<r<+\infty}{h_p(r)}$$
with $p<C(p)<1$. Moreover, $\lim\limits_{p\to-2}C(p)=-2$ and $\lim\limits_{p\to 0}C(p)=1$.
\end{proof}

\begin{lemma}\label{t1l3}
Let $p,q$ be real numbers and $r\in(0,+\infty)$. Let
$$g_{p,q}(r)\equiv\frac{\arsh^{q-1}\,r}{r^{p-1}\sqrt{1+r^2}}.$$

(1) If $p\le-2$, then $g_{p,q}$ is strictly increasing for each $q\ge p$, and $g_{p,q}$ is not monotone for any $q<p$.

(2) If $p>0$, then $g_{p,q}$ is strictly decreasing for each $q\le p$, and $g_{p,q}$ is not monotone for any $q>p$.

(3) If $p=0$, then $g_{p,q}$ is strictly increasing for each $q\ge 1$, $g_{p,q}$ is strictly decreasing for each $q\le 0$,
and $g_{p,q}$ is not monotone for any $0<q<1$.

(4) If $-2<p<0$, then $g_{p,q}$ is strictly increasing for each $q\ge C(p)$, and $g_{p,q}$ is not monotone for any $q<C(p)$.

Here $C(p)$ is the same as in Lemma \ref{t1l2}.
\end{lemma}
\begin{proof} By logarithmic differentiation in $r$,
$$\frac{g'_{p,q}(r)}{g_{p,q}(r)}=\frac{1}{\sqrt{1+r^2}\arsh\, r}[q-h_p(r)],$$
where $h_p(r)$ is the same as in Lemma \ref{t1l2}. Hence the results follow from Lemma \ref{t1l2}.
\end{proof}

\medskip

The following theorem studies the $H_{p,q}$-convexity of $\arsh$.

\begin{theorem}\label{ath1}
The inverse hyperbolic sine function $\arsh$  is strictly $H_{p,q}$-convex on $(0,\infty)$ if and only if $(p,q)\in{D_1}\cup{D_2}$,
while $\arsh$  is strictly $H_{p,q}$-concave on $(0,\infty)$ if and only if $(p,q)\in{D_3}$,
where
$$D_1=\{(p,q)|-\infty<p<-2,\, p\le q<+\infty\},$$
$$D_2=\{(p,q)|-2\le p\le 0,\, C(p) \le q<+\infty\},$$
$$D_3=\{(p,q)|0\le p< +\infty,\,-\infty<q\le p\},$$
and $C(p)$ is the same as in Lemma \ref{t1l2}.
\end{theorem}

\begin{proof}  The proof is divided into the following four cases.

{\bf Case 1.} $p\neq0$ and $q\neq0$.

We may suppose that $0<x\leq y<1$. Define
$$F(x,y)=\arsh^{q}\left(H_p(x,y)\right)-\frac{\arsh^{q}x+\arsh^{q}y}{2}.$$
Let $t=H_p(x,y)$, then $\frac{\partial t}{\partial x}=\frac 12(\frac xt)^{p-1}$. If $x<y$, we see that $t>x$.
By differentiation, we have
$$\frac{\partial F}{\partial x}=\frac{q}{2}x^{p-1}\left(\frac{\arsh^{q-1}t}{t^{p-1}\sqrt{1+t^2}}-\frac{\arsh^{q-1}x}{x^{p-1}\sqrt{1+x^2}}\right).$$

\medskip
{\it Case 1.1.} $p\le-2$, $q\ge{p}$ and $q\neq0$.

By Lemma \ref{t1l3}(1), $\frac{\partial{F}}{\partial{x}}<0$ if $0>q\ge{p}$, and $\frac{\partial{F}}{\partial{x}}>0$ if $q>0  (\ge p)$.
Then $F(x,y)$ is strictly decreasing and $F(x,y)\ge F(y,y)=0$ if $0>q\ge{p}$, and $F(x,y)$ is strictly increasing and $F(x,y)\leq F(y,y)=0$ if $q>0$.
Hence we have
$$\arth(H_p(x,y))\le H_q(\arth x,\arth y)$$
with equality if and only if $x=y$.

In conclusion, $\arsh$ is strictly $H_{p,q}$-convex on $(0,+\infty)$ for $(p,q)\in\{(p,q)|p\le-2,\, p\le q<0\}\cup\{(p,q)|p\le-2,\, q>0\}$.

\medskip
{\it Case 1.2.} $p\le-2$, $q<p$.

By Lemma \ref{t1l3}(1), with an argument similar to Case 1.1, it is easy to see that
$\arsh$ is neither $H_{p,q}$-concave nor $H_{p,q}$-convex on the whole interval $(0,+\infty)$.

\medskip
{\it Case 1.3.} $p>0$, $q\le{p}$ and $q\neq0$.

By Lemma \ref{t1l3}(2), $\frac{\partial{F}}{\partial{x}}>0$ if $q<0(\le{p})$,  and $\frac{\partial{F}}{\partial{x}}<0$ if $0<q\le p$.
Then $F(x,y)$ is strictly increasing and $F(x,y)\le F(y,y)=0$ if $q<0$, and $F(x,y)$ is strictly decreasing and $F(x,y)\ge F(y,y)=0$ if $0<q\le p$.
Hence we have
$$\arth(H_p(x,y))\ge H_q(\arth x,\arth y)$$
with equality if and only if $x=y$.

In conclusion, $\arsh$ is strictly $H_{p,q}$-concave on $(0,+\infty)$ for $(p,q)\in\{(p,q)|p>0,\, q<0\}\cup\{(p,q)|p>0,\, 0<q\le p\}$.

\medskip
{\it Case 1.4.} $p>0$, $q>p$.

By Lemma \ref{t1l3}(2), with an argument similar to Case 1.3, it is easy to see that
$\arsh$ is neither $H_{p,q}$-concave nor $H_{p,q}$-convex on the whole interval $(0,+\infty)$.

\medskip
{\it Case 1.5.} $-2<p<0$, $q\geq{C(p)}$ and $q\neq0$.

By Lemma \ref{t1l3}(4), $\frac{\partial{F}}{\partial{x}}<0$ if $0>q\geq{C(p)}$, and $\frac{\partial{F}}{\partial{x}}>0$ if $q\geq{C(p)}$ and $q>0$.
Then $F(x,y)$ is strictly decreasing and $F(x,y)\ge F(y,y)=0$ if $0>q\geq{C(p)}$, and $F(x,y)$ is strictly increasing and $F(x,y)\leq F(y,y)=0$ if $q\geq{C(p)}$ and $q>0$.
Hence we have
$$\arth(H_p(x,y))\le H_q(\arth x,\arth y)$$
with equality if and only if $x=y$.

In conclusion, $\arsh$ is strictly $H_{p,q}$-convex on $(0,+\infty)$ for $(p,q)\in\{(p,q)|-2<p<0,\, 0>q\ge C(p)\}\cup\{(p,q)|-2<p<0,\, q\geq{C(p)}, q>0\}$.

\medskip
{\it Case 1.6.} $-2<p<0$, $q<{C(p)}$ and $q\neq0$.

By Lemma \ref{t1l2}(4), with an argument similar to Case 1.5, it is easy to see that
$\arsh$ is neither $H_{p,q}$-concave nor $H_{p,q}$-convex on the whole interval $(0,+\infty)$.

\bigskip
{\bf Case 2.} $p\neq0$ and $q=0$.

For $0<x\leq y<1$, let
$$F(x,y)=\frac{\arsh^2(H_p(x,y))}{\arsh{x}\,\arsh{y}},$$
and $t=H_p(x,y)$. If $x<y$, we see that $t>x$.
By logarithmic differentiation, we obtain
$$\frac1{F}\frac{\partial F}{\partial x}=x^{p-1}\left(\frac{(\arsh{t})^{-1}}{t^{p-1}\sqrt{1+t^2}}-\frac{(\arsh{x})^{-1}}{x^{p-1}\sqrt{1+x^2}}\right).$$

\medskip
{\it Case 2.1.} $p\le -2$ and $q=0(>p)$.

By Lemma \ref{t1l3}(1), we have $\frac{\partial F}{\partial x}>0$ and $F(x,y)\le F(y,y)=1$. Hence we have
$$\arsh(H_p(x,y))\le\sqrt{\arsh{x}\,\arsh{y}}$$
with equality if and only if $x=y$.

In conclusion, $\arsh$ is strictly $H_{p,q}$-convex on $(0,+\infty)$ for $(p,q)\in\{(p,q)|-p\le -2, q=0\}$.

\medskip
{\it Case 2.2.} $p>0$ and $q=0(<p)$.

By Lemma \ref{t1l3}(2), we have $\frac{\partial F}{\partial x}<0$ and $F(x,y)\ge F(y,y)=1$. Hence we have
$$\arsh(H_p(x,y))\ge\sqrt{\arsh{x}\,\arsh{y}}$$
with equality if and only if $x=y$.

In conclusion, $\arsh$ is strictly $H_{p,q}$-concave on $(0,+\infty)$ for $(p,q)\in\{(p,q)|p>0, q=0\}$.

\medskip
{\it Case 2.3.} $-2<p<0$ and $q=0\ge C(p)$.

By Lemma \ref{t1l3}(4), we have $\frac{\partial F}{\partial x}>0$ and $F(x,y)\le F(y,y)=1$. Hence
we have
$$\arsh(H_p(x,y))\le\sqrt{\arsh{x}\,\arsh{y}}$$
with equality if and only if $x=y$.

In conclusion, $\arsh$ is strictly $H_{p,q}$-convex on $(0,+\infty)$ for $(p,q)\in\{(p,q)|-2<p<0, q=0\ge C(p)\}$.

\medskip
{\it Case 2.4.} $-2<p<0$ and $q=0<C(p)$.

By Lemma \ref{t1l3}(4), with an argument similar to Case 2.3, it is easy to see that
$\arsh$ is neither $H_{p,q}$-concave nor $H_{p,q}$-convex on the whole interval $(0,+\infty)$.

\bigskip
{\bf Case 3.} $p=0$ and $q\neq0$.

For $0<x\leq y<1$, let
$$F(x,y)=\arsh^q(\sqrt{xy})-\frac{\arsh^q\,x+\arsh^q\,y}{2},$$
and $t=\sqrt{xy}$. If $x<y$, we have that $t>x$. By differentiation, we obtain
$$\frac{\partial F}{\partial x}=\frac{q}{2x}\left(\frac{\arsh^{q-1}t}{t^{-1}\sqrt{1+t^2}}-\frac{\arsh^{q-1}x}{x^{-1}\sqrt{1+x^2}}\right).$$

\medskip
{\it Case 3.1.} $p=0$ and $q\ge 1$.

By Lemma \ref{t1l3}(3), we have $\frac{\partial F}{\partial x}>0$ and $F(x,y)\le F(y,y)=0$. Hence we have
$$\arsh(\sqrt{xy})\leq H_q(\arsh{x},\,\arsh{y})$$
with equality if and only if $x=y$.

In conclusion, $\arsh$ is strictly $H_{p,q}$-convex on $(0,1)$ for $(p,q)\in\{(p,q)|p=0, q\ge 1\}$.

\medskip
{\it Case 3.2.} $p=0$ and $q<0$.

By Lemma \ref{t1l3}(3), we have $\frac{\partial F}{\partial x}>0$ and $F(x,y)\le F(y,y)=0$. Hence we have
$$\arsh(\sqrt{xy})\ge H_q(\arsh{x},\,\arsh{y})$$
with equality if and only if $x=y$.

In conclusion, $\arsh$ is strictly $H_{p,q}$-concave on $(0,1)$ for $(p,q)\in\{(p,q)|p=0, q\le 0\}$.

\medskip
{\it Case 3.3.} $p=0$ and $0<q<1$.

By Lemma \ref{t1l3}(3), with an argument similar to Case 3.1 or Case 3.2, it is easy to see that
$\arsh$ is neither $H_{p,q}$-concave nor $H_{p,q}$-convex on the whole interval $(0,+\infty)$.

{\bf Case 4.} $p=q=0$.

By Case 2.2, for all $x\,,y\in(0,+\infty)$, we have
$$\arsh(H_p(x,y))\ge \sqrt{\arsh{x}\,\arsh{y}},\quad\mbox{for}\quad{p>0}.$$
By the continuity of $H_p$ in $p$ and $\arsh$ in $x$, we have
$$\arsh(H_0(x,y))\ge H_0(\arsh{x},\arsh{y}).$$
In conclusion, $\arsh$ is strictly $H_{0,0}$-concave on $(0,+\infty)$.

This completes the proof of Theorem \ref{ath1}.
\end{proof}

\medskip

Setting $p=1=q$ in Theorem \ref{ath1}, we obtain the concavity of $\arsh$ easily.

\begin{corollary}
The inverse hyperbolic sine function $\arsh$ is strictly concave on $(0,+\infty)$.
\end{corollary}

\section{Proof of Main Results }

\begin{proof}[Proof of Theorem \ref{Lamdd}]

By the proof of Proposition \ref{d3d4for}, we have
$$d_3d_4=\arsh(m\,r)\arsh(m\,r')$$
where $m$, $r$ are the same as in the proof of Proposition \ref{d3d4for}. By Lemma \ref{lecr}(2), we have
$$d_3d_4\le \left(\arsh \left(\frac{\sqrt 2}{2}m\right)\right)^2.$$

This completes the proof of Theorem \ref{Lamdd}.
\end{proof}

\begin{proof}[Proof of Theorem \ref{Lamdad}]
By the proof of Proposition \ref{d3d4for}, we have
$$d_3+d_4=\arsh(m\,r)+\arsh(m\,r')$$
where $m$, $r$ are the same as in the proof of Proposition \ref{d3d4for}. By Lemma \ref{lecr}(3), the desired conclusion follows.
\end{proof}

\medskip

\subsection*{Acknowledgments}I wish to express my sincere gratitude to Professor Matti Vuorinen whose suggestions and ideas were invaluable during my work. This research was partly supported by Turku University Foundation, National Natural Science Foundation of China (NNSF of China, No.11601485) and Science Foundation of Zhejiang Sci-Tech University (ZSTU).

\end{document}